\documentclass[final,3p,times]{elsarticle}

\usepackage{amsmath,amssymb,amsfonts}
\usepackage{amsthm}
\usepackage[dvipsnames,table]{xcolor}
\usepackage{colortbl}
\usepackage{dsfont}
\usepackage{lmodern}
\usepackage{bbm}

\usepackage[labelfont=bf]{caption}
\usepackage{placeins}

\usepackage{tikz}
\definecolor{orange}{HTML}{d9944f}
\usetikzlibrary{patterns,arrows,positioning,fit}
\usetikzlibrary{decorations.pathreplacing}
\usetikzlibrary{decorations.pathmorphing}

\usepackage{graphicx,subcaption,lipsum}
\usepackage{url}
\usepackage{doi}
\usepackage{hyperref}
\usepackage[capitalize]{cleveref}
\usepackage{float}
\usepackage{orcidlink}
\usepackage{booktabs}

\usepackage{multirow}
\usepackage{multicol}

\usepackage{amsthm} 
\newtheorem{theorem}{Theorem}[section]

\newtheorem{definition}[theorem]{Definition}
\newtheorem{assumption}[theorem]{Assumption}
\newtheorem{remark}[theorem]{Remark}

\usepackage{float}
\usetikzlibrary{arrows.meta, positioning, calc}

\usepackage{array}
\usepackage{tabularx}

\newcolumntype{L}[1]{>{\raggedright\arraybackslash}p{#1}}
\newcolumntype{D}{>{\hsize=2.0\hsize\raggedright\arraybackslash}X}
\newcolumntype{V}{>{\hsize=0.70\hsize\centering\arraybackslash}X}

\newcommand{\rowrule}{\cmidrule(lr){1-6}}

\usepackage{lineno}

\usepackage[figuresright]{rotating}
\usepackage{orcidlink}

\usepackage[ruled,vlined]{algorithm2e}

\usepackage{todonotes}
\newboolean{SHOWCOMMENTS}
\setboolean{SHOWCOMMENTS}{true}
\ifthenelse{\boolean{SHOWCOMMENTS}}{
\newcommand{\tdHZ}[1]
{\todo[color=green!25,inline]{\footnotesize{\bf Henrik:} #1}}
}{
\newcommand{\tdHZ}[1]{}
}

\ifthenelse{\boolean{SHOWCOMMENTS}}{
\newcommand{\tdMK}[1]
{\todo[color=blue!25,inline]{\footnotesize{\bf Martin:} #1}}
}{
\newcommand{\tdMK}[1]{}
}

\renewcommand{\d}[1]{\mathrm{d}#1}

\newcommand{\vismode}{showpapertwo}
\newcommand{\papertwo}[1]{%
  \ifx\vismode\showpapertwo
  \else
    {\color{blue} #1}%
  \fi
}

\begin{document}

\begin{frontmatter}

\title{Efficient numerical computation of traveler states in explicit mobility-based metapopulation models: Mathematical theory and application to epidemics}

\author[a]{Henrik Zunker\orcidlink{0000-0002-9825-365X}\corref{cor1}}
\cortext[cor1]{Corresponding author}
\ead{Henrik.Zunker@DLR.de}
\author[a]{Ren\'e Schmieding \orcidlink{0000-0002-2769-0270}}
\author[b]{Jan Hasenauer\orcidlink{0000-0002-4935-3312}}
\author[a,b]{Martin J. Kühn\orcidlink{0000-0002-0906-6984}}
\ead{Martin.Kuehn@DLR.de}

\address[a]{Institute of Software Technology, Department of High-Performance Computing,
German Aerospace Center, Cologne, Germany}
\address[b]{Bonn Center for Mathematical Life Sciences and Life and Medical Sciences Institute, University of Bonn, Germany}

\begin{abstract}
Metapopulation models are powerful tools for capturing the spatio-temporal spread of infectious diseases. Models that explicitly account for traveler origins and destinations, such as Lagrangian metapopulation models, enable a detailed representation of mobility and traveling subpopulations. However, tracking these subpopulations requires a separate set of ordinary differential equations (ODEs) for each traveler group, leading to quadratic growth in system size with the number of spatial patches in densely connected networks. While specific approaches reducing the effort of traveler state estimation have been proposed, these approaches are either model-specific or heuristic and lack sufficient numerical accuracy.

Here, we introduce a Runge-Kutta (RK) stage-aligned computation of traveler states that leverages the precomputed intermediate stage values of explicit RK methods under the assumption of localized homogeneous mixing. We prove that the resulting numerical solution is identical to that of the standard Lagrangian formulation when solved with the corresponding RK method. For compartments without inflows, we further show that the exact same results can be obtained using a simple algebraic scaling based on the initial traveler share.

When embedded in a recently proposed metapopulation framework that combines local ODE dynamics on nodes with discrete mobility events along network edges, the stage-aligned approach eliminates the need for heuristic traveler approximations. In contrast to the standard Lagrangian formulation, the resulting method enables efficient simulations by reducing the global ODE system to linear scaling in the number of patches, while the remaining quadratic interactions are handled through simple and highly efficient algebraic updates for traveler estimation. Numerical experiments confirm the theoretical results, demonstrating optimal convergence order and good scalability with respect to the number of patches. Benchmarks on fully connected networks with up to 1025 patches, 1024 local travel connections, and six age groups achieve speedups of up to 76 and 50 for first- and fourth-order Runge–Kutta methods, respectively.
\end{abstract}

\begin{highlights}
    \item Introduction of a Runge-Kutta (RK) stage-aligned computation for traveler states in metapopulation models with explicit mobility.
    \item Theoretical proof of identical numerical results as in the standard Lagrangian metapopulation model.
    \item Reduction of ODE system size from quadratic to linear scaling with the number of patches and populations.
    \item Highly efficient computation of quadratically scaling interactions terms. 
    \item Numerical benchmarks show up to 76x (RK-1) and 50x (RK-4) speedups on fully coupled networks. 
\end{highlights}

\begin{keyword}
Metapopulation Model \sep Compartmental Model \sep Mobility \sep Runge-Kutta methods \sep MEmilio framework \sep Infectious disease modeling

\MSC[2020] 34-04 \sep 65L05 \sep 65Y20
\end{keyword}

\end{frontmatter}

\section{Introduction}
\label{sec:introduction}

Over the past decades, many methods have been proposed for capturing spatio-temporal dynamics of mobility and infectious disease spread. In agent-based models~\cite{kerr_covasim_2021,KERKMANN2025110269}, each individual is modeled as an agent with a potentially a large number of features. For large populations and with linear to quadratic complexity in the number agents, the computational effort of agent-based models can become considerable. A common and mathematically tractable alternative in spatial epidemic models is based on systems of ODEs, which remain popular due to their conceptual clarity and interpretability. In ODE-based metapopulation models, the total population is considered to be split over different patches; which can be communities, districts, regions, or entire states. Then, the spread of infectious disease within the individual patches is governed by a compartmental model, while mobility links the patches. Assessing and predicting the spread of infectious diseases in spatially structured populations is especially relevant when interventions are implemented regionally or when heterogeneous outbreak dynamics emerge across administrative units such as states, counties, or municipalities. By resolving disease dynamics across multiple interconnected patches, metapopulation models enable the explicit representation of spatial heterogeneity, regional differences in transmission conditions, and human mobility patterns. Beyond infectious disease epidemiology, metapopulation models have been applied for instance to analyze the behavior of multi-patch logistic population systems with migration~\cite{10.3934/dcdsb.2021025}, to study control and elimination strategies in spatially structured pest populations~\cite{BLIMAN2025116047}, and to derive optimal harvesting strategies in predator-prey metapopulation models~\cite{supriatna1999harvesting}.

For infectious disease dynamics, a large body of work has demonstrated that neglecting mobility or spatial structure can substantially bias outbreak size estimates and reproduction numbers, particularly in highly connected settings~\cite{wesolowski2012quantifying,bajardi2011human,balcan2009multiscale,kraemer2020mobility,doi:10.1073/pnas.0510525103}. When mobility acts as a driver of transmission or becomes a central part of the modeling task, it is important to choose the right metapopulation design with respect to the mobility processes. While some approaches model mobility only implicitly, including infected subpopulations from neighboring patches in the force of infection without physically relocating individuals~\cite{wang2018,liu_modelling_2022}, explicit mobility models physically transfer individuals between patches, thereby altering the local population states. We can distinguish two main approaches for explicit mobility: Eulerian and Lagrangian formulations~\cite{Citron2021}.

In Eulerian metapopulation models, individuals are transferred between patches via mobility and immediately become part of the destination patch's population; their origin is no longer tracked. Mobility is typically represented through continuous flows. If a single-patch model contains $N_C$ compartments and the system consists of $N_P$ patches, then the total number of ODE states scales linearly as $\mathcal{O}(N_P N_C)$. Origin--destination mobility flux data, representing average movements per unit time between patch pairs, are usually sufficient to parametrize such models~\cite{Citron2021,Peirlinck2020,MENG2022489}. While computationally efficient, Eulerian models do not distinguish residents from visitors once movement occurs, resulting in artificial mixing of populations.

In Lagrangian metapopulation models, individuals retain an association with their patch of origin while traveling~\cite{ArinoVdD2003a,COSNER2009550,Arino2015,Rapaport2022}. Accordingly, these models do not rely on the assumption that populations in the same patch share the same distribution of infection states, which allows for a more coherent description of the impact of mobility, e.g., traveling and commuting. Traveling subpopulations are tracked by introducing separate compartments for each ordered pair of origin and destination patches. Consequently, if patches are fully connected and each local model contains $N_C$ states, the total number of ODE states scales quadratically as $\mathcal{O}(N_P^2 N_C)$. Hence, while Lagrangian metapopulation models enable precise tracking of infection state changes under destination conditions, they also substantially increase dimensionality and computational cost.

Eulerian and Lagrangian metapopulation models can predict substantially different outbreak sizes and reproduction numbers~\cite{Citron2021, VargasBernal2022}. Since Lagrangian formulations explicitly track travelers experiencing local transmission conditions at their destinations, they are conceptually advantageous and can provide more realistic descriptions. However, to reduce the computational burden, various numerical approximations of their solutions have been proposed. In particular, operator splitting outside the main ODE integration~\cite{Lipshtat2021} and auxiliary computation steps~\cite{kuhn_assessment_2021} have been proposed to handle reverse traveling. These approaches yield approximations which scale only as $\mathcal{O}(N_P N_C)$ and are effective in certain regimes. Yet, they are often model dependent, may rely on restrictive step-size conditions, or remain largely heuristic.

In this work, we suggest an efficient approach leveraging explicit Runge-Kutta structures to compute states of traveling subpopulations on-the-fly with the numerical solution of a patch-aggregated system. We theoretically prove that the corresponding result is identical to the numerical solution of the standard Lagrangian system. While traveler computation of our novel approach also scales with $\mathcal{O}(N_P^2)$, the global numerical approach, with a reduced ODE system of size $\mathcal{O}(N_P)$, results in substantially reduced workload and thus runtimes when compared to the standard Lagrangian formulation.

\section{ODE-based compartmental models}
\label{sec:single}

The basis of our work are ODE-based compartmental models. In the simplest case, a single population is divided into $N_C$ different compartments with $N_C\in\mathbb{N}$. Let $x_k(t):\mathbb{R}\rightarrow \mathbb{R}_{\geq 0}$ denote the number of individuals in the $k$-th compartment at time $t$, also referred to as compartment size. Then the system state at time $t$ is given by the vector
\begin{align*}
    x(t) = (x_1(t), \dots, x_{N_C}(t))^T \in \mathbb{R}_{\ge 0}^{N_C}.
\end{align*}
Compartments may differ with respect to infection status (e.g., susceptible, exposed, infectious, recovered) as well as demographic attributes such as age or risk group. However, individuals within the same compartment are assumed to be indistinguishable.

The time evolution of the compartment sizes is determined by a set of $N_T\in\mathbb{N}$ transition processes. The effect of these transitions on the compartment vector is encoded in a matrix
\begin{align*}
B \in \mathbb{R}^{N_C \times N_T},
\end{align*}
whose columns represent the net change in compartments caused by a single transition. Specifically, if transition $r$ moves individuals from source compartment $j$ to target compartment $k$, then the $r$-th column of $B$ satisfies
\begin{align*}
B_{jr}=-1,\qquad B_{kr}=+1,
\end{align*}
and $B_{ir}=0$ for all $i\notin\{j,k\}$. The transition rates (flows) are collected in a vector
\begin{align*}
{f}(x(t),t) = \bigl(f_1(x(t),t), \dots, f_{N_T}(x(t),t)\bigr)^\top
\in \mathbb{R}_{\ge 0}^{N_T}.
\end{align*}
The rate of the $r$-th transition is assumed to take the form
\begin{align}
f_r(x(t),t) = x_{j(r)}(t)\,d_r(x(t),t),
\end{align}
where $j(r)$ is the index of the source compartment of the transition and $d_r(x,t)\ge 0$ is a continuous function. This assumption holds for a broad range of (epidemiological) models, including SIR- and SIS-type structures. Depending on the structure of $d_r$, we distinguish two types of processes:
\begin{itemize}
    \item \textbf{Interaction independent processes:} If the function $d_r(x,t)$ is independent of $x$, the transition depends only on the source compartment size (e.g., recovery, progression, waning).
    \item \textbf{Interaction dependent processes:} If the function depends on $x$, there is a direct or indirect interaction with other compartments (e.g., infection transmission).
\end{itemize}

Using the matrix $B$ and the flow vector $f$, we write the dynamics of compartmental models compactly as
\begin{align}
\label{eq:stoich_form}
\frac{\d x(t)}{\d t} = B\, f(x(t),t).
\end{align}
We assume that $f(\cdot,t)$ is sufficiently regular such that the initial value problem defined with~\cref{eq:stoich_form} is well posed and preserves nonnegativity of the state. Defining the matrix $D(x(t),t)\in\mathbb{R}^{N_T\times N_C}$ with entries
\begin{align}
\label{eq:matrix_D}
D_{rj}(x(t),t)
=
\begin{cases}
d_r(x(t),t), & \text{if } j=j(r),\\
0, & \text{otherwise},
\end{cases}
\end{align}
we can write~\cref{eq:stoich_form} as
\begin{align}
\label{eq:form_MDx}
\frac{\d x(t)}{\d t}=B \, D(x(t),t)\,x(t),
\end{align}
where $B \, D(x(t),t)\in\mathbb{R}^{N_C\times N_C}$.

\section{ODE-based Lagrangian metapopulation model}
\label{sec:metapop}
\subsection{Standard formulation with continuous mobility}
\label{sec:standard_lagrange}

We consider a spatial domain $\Omega \subset \mathbb{R}^2$ partitioned into $N_P$ patches. The population present in patch $p$ at time $t$ may consist of residents of patch $p$ (i.e., individuals whose home patch is $p$) as well as travelers from other patches. Following the Lagrangian metapopulation formulation, we denote by $x^{(p;q)}(t) \in \mathbb{R}_{\ge 0}^{N_C}$ the compartment sizes of individuals whose patch of residence is $p$ and who are currently located in patch $q$. Accordingly, the state of the resident population of patch $p$ is given by
\begin{align}
x^{(p)}(t) = \sum_{q} x^{(p;q)}(t),
\end{align}
while the total population currently present in patch $q$ is
\begin{align}
\widetilde{x}^{(q)}(t) = \sum_{p} x^{(p;q)}(t).
\end{align}
For convenience, we introduce the vector $\hat{x}^{(p)}(t)$, which fully characterizes the state of the residents of patch $p$,
\begin{align}
\label{eq:disagg}
\hat{x}^{(p)}(t)=
\begin{pmatrix}
    x^{(p;1)}(t)\\
    \vdots\\
    x^{(p;N_P)}(t)
\end{pmatrix}.
\end{align}

\begin{assumption}\label{ass:mixing}
For each compartment, we assume \textbf{homogeneous mixing}. That is, individuals in the same compartment are indistinguishable, and any effect acting on a subpopulation in this compartment affects all individuals identically.
\end{assumption}

Under Assumption~\ref{ass:mixing}, the dynamics of $x^{(p;q)}$ for any pair of patches $(p,q)$, representing the population living in $p$ and currently being in $q$, can be written as
\begin{align} \label{eq:MasterEquation}
\frac{\d x^{(p;q)}(t)}{\d t} = G^{(q)}(x^{(p;q)}(t), \widetilde{x}^{(q)}(t), t) + M^{(p,q)}(\hat{x}^{(p)}(t),t).
\end{align}
Here, $G^{(q)}:\mathbb{R}^{N_C}_{\geq 0}\times\mathbb{R}^{N_C}_{\geq 0}\times\mathbb{R} \to \mathbb{R}^{N_C}$ is the (local) transition operator and $M^{(p,q)}$ is the (local) mobility operator.

The transition operator $G^{(q)}$ describes state transitions occurring under the conditions in the current patch $q$. Under the homogeneous mixing assumption, flows scale linearly with the size of the source compartment and depend only on the characteristics of the total population currently present in patch $q$. The transition operator therefore takes the form
\begin{align}
\label{eq:DtoDp}
G^{(q)}(x^{(p;q)}(t), \widetilde{x}^{(q)}(t), t)
=
B\,D^{(q)}(\widetilde{x}^{(q)}(t),t)\,x^{(p;q)}(t).
\end{align}
The matrix $D^{(q)}(\widetilde{x}^{(q)}(t),t)$ collects the transition rates, which may vary between patches to account for patch-specific characteristics such as contact patterns or population density.

The mobility operator $M^{(p,q)}$ describes the travel of individuals residing in patch $p$. Travel rates from $p$ to $q$ are denoted by $m^{(p;p\to q)}(t)$ and return rates by $m^{(p;q\to p)}(t)$, allowing for varying rates according to the compartment of the individuals. With $\odot$ denoting element-wise multiplication, the mobility operator for $q \neq p$ is
\begin{align}
\label{eq:mob_pq}
M^{(p,q)} = m^{(p;p\to q)}(t) \odot x^{(p;p)}(t) - m^{(p;q\to p)}(t) \odot x^{(p;q)}(t),
\end{align}
where the first term represents the travelers arriving from $p$ in $q$ and the second term represents the travelers leaving $q$ for $p$. For $q=p$, i.e., the subpopulation $x^{(p;p)}$ living and staying in patch $p$, we have
\begin{align}
\label{eq:mob_pp}
M^{(p,p)} = - \sum_{q} m^{(p;p\to q)}(t) \odot x^{(p;p)}(t) + \sum_{q} m^{(p;q\to p)}(t) \odot x^{(p;q)}(t),
\end{align}
where the first sum represents all outbound and the second sum all inbound travelers, respectively. Summed over all contributions, we have mass conservation as given by
\begin{align}
M^{(p,p)}(t)
=
- \sum_{q \neq p} M^{(p,q)}(t).
\end{align}

We denote the ODE system of dimension $N_P^2\, N_C$ given by~\cref{eq:MasterEquation,eq:DtoDp,eq:mob_pq,eq:mob_pp}, which explicitly tracks the dynamics of traveling subpopulations between all pairs of patches, as the \textit{standard Lagrangian formulation} of metapopulation models.

\subsection{Standard formulation with discrete mobility}

In many applications, mobility occurs in discrete or nearly discrete events, such as daily commuting or scheduled travel. This motivates modeling travel as instantaneous exchanges between patches occurring at prescribed time points 
\begin{align}
    \label{eq:global_exchange_times}
    \begin{aligned}
        0 = t_0 < t_1 < t_2 < \dots.
    \end{aligned}
\end{align}
Travel and return rates are in this case given by sums of Dirac delta function, $\delta(t)$, yielding
\begin{align}\label{eq:discrete_mob}
m^{(p;p\to q)}(t) = \sum_{a\in\mathbb{N}_0} \mu^{(p;p\to q)}_a\delta(t-t_a) \quad\text{and}\quad
m^{(p;q\to p)}(t) = \sum_{a\in\mathbb{N}_0} \mu^{(p;q\to p)}_a\delta(t-t_a),
\end{align}
with $\mu^{(p;p\to q)}_a, \mu^{(p;q\to p)}_a\in[0,1]$ denoting the magnitude of instantaneous mobility exchange at $t_a\in\mathbb{R}_{\geq 0}$.

In the case of this discrete mobility formulation~\cref{eq:discrete_mob}, \cref{eq:MasterEquation} simplifies within the interval $I_a=(t_a,t_{a+1})$, $a\in\mathbb{N}_0$ to
\begin{align} \label{eq:MasterEquation_wo_mob}
\frac{\d x^{(p;q)}(t)}{\d t} = B\,D^{(q)}(\widetilde{x}^{(q)}(t),t) x^{(p;q)}(t).
\end{align}
As traveling subpopulations are explicitly described in~\cref{eq:MasterEquation_wo_mob}, the initial value problem which needs to be solved on $t\in I_a$ scales in densely connected system as $\mathcal{O}(N_P^2)$. 

\subsection{A graph approach for aggregated dynamics with decoupled traveler-state approximation}
\label{sec:graphode}

In order to avoid quadratically growing ODE systems, a graph-based approach with decoupled traveler-state estimation was suggested in~\cite{kuhn_assessment_2021} and extended in~\cite{zunker_novel_2024}. Here, the dynamics of the aggregated population in a patch $q$ (represented by a graph node) and solved with the system
\begin{align} \label{eq:graph_ode_2}
\frac{\d \widetilde{x}^{(q)}(t)}{\d t} = B\,D^{(q)}(\widetilde{x}^{(q)}(t),t) \widetilde{x}^{(q)}(t)
\end{align}
that only scales linearly with the number of patches. For an interval $I_a=(t_a,t_{a+1})$ of length $h$, traveler states are heuristically estimated through an auxiliary step with the explicit Euler method as
\begin{align}\label{eq:auxeuler}
x^{(p;q)}(t_{a+1}) = x^{(p;q)}(t_{a}) + h\,B\,D^{(q)}(\widetilde{x}^{(q)}(t_{a}), t_{a})x^{(p;q)}(t_{a}).
\end{align}
The heuristic in~\cref{eq:auxeuler} actually mean that the population living and staying at patch $q$ only serve as contact population without individual state transitions while traveler states are approximated with an first-order explicit Euler step. While the formulation in~\cref{eq:graph_ode_2} is attractive for its reduced system size, heuristic traveler state approximations yield particular problems when available subpopulations are overestimated and, generically, we can not expect~\cref{eq:auxeuler} to align with~\cref{eq:MasterEquation_wo_mob}.

\section{An efficient numerical approach for the computation of traveler states}
\label{sec:stage-aligned}

\subsection{Stage-aligned Runge-Kutta method}

Based on Assumption~\ref{ass:mixing} and the discrete mobility formulation, we now introduce a numerical approach for computing traveler states on-the-fly with a Runge-Kutta method applied to the patch-level aggregated dynamics. This approach combines the advantages of solving a system for the aggregated population in a patch as introduced by~\cite{kuhn_assessment_2021} (cf.~\cref{sec:graphode}) with the numerical computation of $x^{(p;q)}(t)$, entirely avoiding the explicit solution of the standard Lagrangian formulation of~\cref{eq:MasterEquation_wo_mob}. We consider a series of mobility event timings~\cref{eq:global_exchange_times} that generalizes the approach by~\cite{kuhn_assessment_2021} and also allows for flexible traveling and long-term travel between two different patches.

We first provide one property for the analytical solution of compartments without inflows, which will later be used to simplify the numerical scheme.

\begin{theorem}
\label{thm:pure_outflow}
Let Assumption~\ref{ass:mixing} hold true. Consider $t_a\in\mathbb{R}$ with the next mobility event timing at $t_{a+1}\in\mathbb{R}$ and a compartment $j \in \{1, \dots, N_C\}$ with no inflows. Then, the population share
\begin{align}
\label{eq:shares}
    \xi^{(p;q)}_j(t)=\frac{x^{(p;q)}_j(t)}{\widetilde{x}^{(q)}_j(t)}, \quad \text{with}\quad\xi_j^{(p;q)}(t)=0\quad\text{if}\quad\widetilde{x}^{(q)}_j(t)=0,
\end{align}
is constant for $t\in[t_a,t_{a+1})$ and it holds that
\begin{align}
\label{eq:computesharexpq}
    x^{(p;q)}_j(t)=\xi^{(p;q)}_j(t_a)\,\widetilde{x}^{(q)}_j(t).
\end{align}
\end{theorem}

\begin{proof}
If $\widetilde x^{(q)}_j(t)=0$, the share is zero by definition. Assume $\widetilde x^{(q)}_j(t)>0$.
Based on the formulations in~\cref{eq:MasterEquation_wo_mob} and~\cref{eq:graph_ode_2}, the change in compartment $j$ is determined by the transitions connected to it. Since compartment $j$ has no inflows, only outflow transitions $r$ with source compartment $j(r) = j$ contribute.
Let $\nu_j(\widetilde{x}^{(q)}(t), t) = \sum_{r: j(r)=j} d_r(\widetilde{x}^{(q)}(t), t)$ denote the total outflow intensity per individual in compartment $j$. Thus, the derivatives for the aggregated population $\widetilde{x}^{(q)}$ and the subpopulation ${x}^{(p;q)}$ are
\begin{align*}
    \frac{\d \widetilde x^{(q)}_j(t)}{\d t} &= - \nu_j(\widetilde{x}^{(q)}(t), t) \, \widetilde x^{(q)}_j(t), \\
    \frac{\d x^{(p;q)}_j(t)}{\d t} &= - \nu_j(\widetilde{x}^{(q)}(t), t) \, x^{(p;q)}_j(t).
\end{align*}
Applying the quotient rule for the time derivative of $\xi^{(p;q)}_j(t)$ yields:
\begin{align*}
\frac{\d \xi^{(p;q)}_j(t)}{\d t}
&= \frac{\frac{\d}{\d t}(x^{(p;q)}_j(t))\,\widetilde x^{(q)}_j(t) - x^{(p;q)}_j(t)\,\frac{\d}{\d t}(\widetilde x^{(q)}_j(t))}  {\bigl(\widetilde x^{(q)}_j(t)\bigr)^2}\\
&=\frac{\bigl(- \nu_j(\widetilde{x}^{(q)}(t), t) x^{(p;q)}_j(t)\bigr)\widetilde x^{(q)}_j(t) - x^{(p;q)}_j(t)\bigl(- \nu_j(\widetilde{x}^{(q)}(t), t) \widetilde x^{(q)}_j(t)\bigr)}{(\widetilde x^{(q)}_j(t))^2} = 0.
\end{align*}
Hence, $\xi^{(p;q)}_j(t)$ is constant, specifically $\xi^{(p;q)}_j(t) = \xi^{(p;q)}_j(t_a)$.
\end{proof}

For any traveling subpopulation $x^{(p;q)}$ ($p \neq q$) present in patch $q$, all function evaluations $B\,D^{(q)}(\widetilde{x}^{(q)}(t),t)$ of a Runge-Kutta scheme applied to~\cref{eq:MasterEquation_wo_mob} act identically on all subpopulations (cf. homogeneous mixing assumption) and are obtained with the Runge-Kutta scheme applied to~\cref{eq:graph_ode_2}. Definition~\ref{def:adaptedrk} formalizes the Runge-Kutta \textit{stage-aligned} computation of traveler groups. \cref{fig:graph_ode_flow_2} visualizes the corresponding piecewise-continuous simulation scheme.

\begin{figure}[!t]
    \centering
    \includegraphics[width=\textwidth]{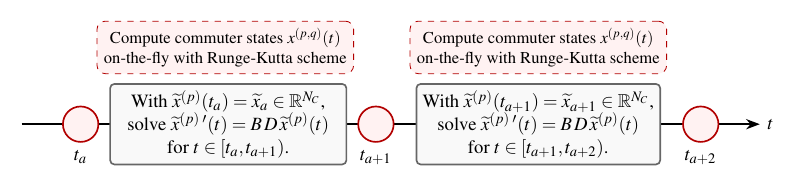}
    \caption{\textbf{Visual representation of the generalized piecewise-continuous simulation scheme.} During intervals $I_{a}$, the system evolves with all individuals in aggregated dynamics while traveler states are computed on-the-fly with the Runge-Kutta scheme to be used for the aggregated dynamics.}
    \label{fig:graph_ode_flow_2}
\end{figure}

\begin{definition}[Runge-Kutta stage-aligned traveler computation]
\label{def:adaptedrk}
Let an explicit single-step Runge-Kutta method with $S$ stages (denoted RK-$S$), $t_a\in\mathbb{R}$, step size $h>0$, and Butcher tableau coefficients $(a_{sr},b_s,c_s)$ for $s,r \in \{1,\dots,S\}$ be given. Assume that we have already precomputed (and stored) through the RK-$S$ scheme applied to the aggregated formulation the stage-based values
\begin{align}
\label{eq:precompxq}
\widetilde{x}^{(q),(s)}\quad\text{and}\quad D^{(q),(s)} = D(\widetilde{x}^{(q),(s)},\, t_a + c_s h).
\end{align}
We define the stage-aligned computation of traveling subpopulations $x^{(p;q)}$ for $p\in\{1,\ldots,N_P\}$ and $s\in \{1,\dots,S\}$ through the following steps.
\begin{enumerate}[i)]
    \item We compute the subpopulation stage state $x^{(p;q),(s)}$ using initial values and previous stage derivatives $k^{(p;q),(r)}$:
    \begin{align}
            \label{eq:rk_traveler_stages}
            x^{(p;q),(s)} = x^{(p;q)}(t_a) + h \sum_{r=1}^{s-1} a_{sr}\,k^{(p;q),(r)},
    \end{align}
    where for $s=1$ the sum is empty and $x^{(p;q),(1)} = x^{(p;q)}(t_a)$.
    \item Using the precomputed intensities of the sparse matrices $D^{(q),(s)}$, we compute the flows
    \begin{align}
    \label{eq:stage_f}
    f^{(p;q),(s)}=D^{(q),(s)}x^{(p;q),(s)}.
    \end{align}
    \item Taking the sparsity structure of $B$ into account, we compute the stage-specific derivatives determined by
    \begin{align}
        \label{eq:rk_stage_derivative}
        \begin{aligned}
        k^{(p;q),(s)} &= B \, f^{(p;q),(s)}.
        \end{aligned}
    \end{align}
\end{enumerate}

The final update is set as
\begin{align}
    \label{eq:rk_final_update}
    x^{(p;q)}(t_{a}+h) = x^{(p;q)}(t_a)+h\sum_{s=1}^S b_s\,k^{(p;q),(s)}.
\end{align}
\end{definition}

\begin{remark}
Note that the stage-aligned traveler-state computation can be given in two formulations, either by precomputation of $\widetilde{x}^{(q)}$ as in~\cref{eq:precompxq} and subsequent post-computing of all traveler groups as in~\cref{eq:rk_traveler_stages} or with the true on-the-fly computation of all traveler states inside the single adapted scheme, i.e., by computing $\widetilde{x}^{(q)}$ from~\cref{eq:rk_traveler_stages} and between~\cref{eq:rk_traveler_stages} and~\cref{eq:stage_f}. Note the formulation with precomputed aggregated values has, in the case of adaptive time stepping, the advantage that travelers are only computed if the time step size is not rejected, the formulation could theoretically lead to slightly different time steps as the time step acceptance criterion is evaluated solely on the aggregated values.
\end{remark}

\begin{theorem}
\label{thm:rk_discrete_equivalence}
Let Assumption~\ref{ass:mixing} hold true. Let the patch-aggregated populations be advanced by an RK-$S$ scheme with step size $h>0$, Butcher tableau coefficients $(a_{sr},b_s,c_s)$ for $s,r \in \{1,\dots,S\}$, and convergence order $q$. Then, for every step, the solution $x^{(p;q)}$ obtained from Definition~\ref{def:adaptedrk} is identical to the solution of the RK-$S$ scheme applied to the standard Lagrangian formulation in~\cref{eq:MasterEquation_wo_mob}. In particular, the solution from Definition~\ref{def:adaptedrk} has the same convergence order.
\end{theorem}

\begin{proof}
Consider a single step of the RK-$S$ method. Let $x^{(p;q)}_{\text{lagr}}$ denote the numerical solution obtained from the standard Lagrangian system in~\cref{eq:MasterEquation_wo_mob} and $x^{(p;q)}$ the stage-aligned traveler state computation from Definition~\ref{def:adaptedrk}. We show by induction over the stages $s=1,\dots,S$ that the additionally computed stage derivatives $k^{(p;q),(s)}$ match with the stage derivatives $k^{(p;q),(s)}_{\text{lagr}}$ of the standard Lagrangian formulation.

\underline{$s=1$:}
For an explicit Runge-Kutta scheme, the first stage corresponds to the state at $t_a$. Thus, $x^{(p;q),(1)} = x^{(p;q)}(t_a) = x^{(p;q),(1)}_{\text{lagr}}$.
The aggregated stage value is $\widetilde{x}^{(q),(1)} = \sum_{p} x^{(p;q)}(t_a) = \widetilde{x}^{(q)}(t_a)$.
The derivative in the stage-aligned computation is
\begin{align*}
    k^{(p;q),(1)} = B\, D(\widetilde{x}^{(q),(1)}, t_a) x^{(p;q),(1)}.
\end{align*}
For the standard Lagrangian system, the evaluation of the derivative at the initial stage ($s=1$) is determined by evaluating the transition rates using the sum of all subpopulations. Because both systems share the same initial conditions, this sum equals the aggregated stage, $\sum_p x^{(p;q),(1)}_{\text{lagr}} = \widetilde{x}^{(q),(1)}$. Thus, both approaches construct the exact same intensity matrix $D$, leading to identical subpopulation derivatives.
Applying the matrix representation to the Lagrangian system yields
\begin{align*}
    k^{(p;q),(1)}_{\text{lagr}} = B\, D(\widetilde{x}^{(q),(1)}, t_a) x_{\text{lagr}}^{(p;q),(1)},
\end{align*}
and, thus, $k^{(p;q),(1)} = k_{\text{lagr}}^{(p;q),(1)}$.

\underline{$s-1 \to s$:}
Assume $k^{(p;q),(r)} = k_{\text{lagr}}^{(p;q),(r)}$ for all $r < s$ and all subpopulations residing in $p\in\{1,\ldots,N_P\}$. 

First, we observe that the computed stage states $x^{(p;q),(s)}$ from \cref{eq:rk_traveler_stages} use the same linear combination of previous derivatives as the Lagrangian system's stage state $x_{\text{lagr}}^{(p;q),(s)}$. By the induction hypothesis, it follows $x^{(p;q),(s)} = x_{\text{lagr}}^{(p;q),(s)}$.

Next, we verify the consistency of the aggregated values. Summing all local contributions yields
\begin{align*}
    \sum_{p} x^{(p;q),(s)}
    = \sum_{p} \left( x^{(p;q)}(t_a) + h \sum_{r=1}^{s-1} a_{sr} k^{(p;q),(r)} \right).
\end{align*}
Due to the linearity of the system, the sum of the contributions $k^{(p;q),(r)}$ equals the evaluation $\widetilde{k}^{(q),(r)}$ of the aggregated system for $r\in\{1,\ldots,s-1\}$, and, thus, $\sum_{p} x^{(p;q),(s)} = \widetilde{x}^{(q),(s)}$.

Finally, we compare the evaluations of the right-hand side of the current stage.
We compute $D^{(q),(s)}$ based on $\widetilde{x}^{(q),(s)}$ and set
\begin{align*}
    k^{(p;q),(s)} = B\, D^{(q),(s)} x^{(p;q),(s)}.
\end{align*}
The Lagrangian system computes rates based on its internal aggregations $\sum_p x_{\text{lagr}}^{(p;q),(s)}$, for which we have equality to $\widetilde{x}^{(q),(s)}$, and applies them to its state $x_{\text{lagr}}^{(p;q),(s)}$:
\begin{align*}
    k_{\text{lagr}}^{(p;q),(s)}
    = B\, D\Bigl(\widetilde{x}^{(q),(s)}, t_a + c_s h\Bigr) x_{\text{lagr}}^{(p;q),(s)}
    = B\, D^{(q),(s)} x^{(p;q),(s)}.
\end{align*}
Thus, $k_{\text{lagr}}^{(p;q),(s)} = k^{(p;q),(s)}$.
Since all stage derivatives are identical, the final update $x^{(p;q)}(t_{a}+h)$ in~\cref{eq:rk_final_update} is identical to the update of the standard Lagrangian formulation.
\end{proof}

While Definition~\ref{def:adaptedrk} provides a general scheme to compute states of traveling subpopulations, \cref{thm:pure_outflow} implies that for compartments without inflows the stage-aligned computation can be replaced by a simple algebraic scaling without introducing errors associated with the computation steps. This allows for a combined numerical strategy where the stage-aligned computations are avoided for compartments without inflow. Substituting the Runge-Kutta updates with this simple scaling for compartments without inflows preserves the identical numerical solution.

\subsection{Comparison of computational complexity}

The standard RK discretization of the Lagrangian formulation (cf.~\cref{eq:MasterEquation_wo_mob}) requires solving a global ODE system whose dimension grows quadratically with the number of patches in densely connected networks. In contrast, the proposed stage-aligned Runge-Kutta scheme, embedded in the graph structure for aggregated dynamics, avoids evaluating the expensive right-hand side of the full Lagrangian system for all traveler subpopulations. Instead, the Runge-Kutta stages are computed only for the aggregated patch states, while traveler states are updated through algebraic transformations aligned with the Runge-Kutta stages. This reduces the dimension of the ODE integration from $\mathcal{O}(N_P^2)$ to $\mathcal{O}(N_P)$, while shifting the $\mathcal{O}(N_P^2)$ scaling entirely to these simple and efficient algebraic updates.

The proposed approach possesses the same scaling characteristic as the previously introduced auxiliary Euler heuristic~\cite{kuhn_assessment_2021}. That method evaluates the right-hand side only for aggregated patch states and updates traveler states through auxiliary Euler steps, resulting in an identical split of linear ODE scaling and quadratic scaling with respect to traveler-state computations. However, the auxiliary Euler approach provides only an approximation to the numerical solution produced by the standard Lagrangian formulation, whereas the proposed stage-aligned scheme remains consistent with the underlying Runge-Kutta discretization.

A summary of the resulting per-step computational complexities is given in \cref{tab:costs}.

\begin{table}[H]
  \centering
  \caption{\textbf{Per-step computational complexity in a dense network with $N_P$ patches.}}  
  \begin{tabular}{@{}lll@{}}
    \toprule
    \textbf{Method} & \textbf{Global ODE system} & \textbf{Traveler-state computation}\\
    \midrule
    Standard Lagrangian & $\mathcal{O}(N_P^2)$ & -- (integrated in global system) \\
    Auxiliary Euler (w/ Graph) & $\mathcal{O}(N_P)$ & $\mathcal{O}(N_P^2)$ (auxiliary explicit Euler step) \\
    Stage-aligned (w/ Graph)& $\mathcal{O}(N_P)$ & $\mathcal{O}(N_P^2)$ (algebraic computation inside Runge-Kutta) \\
    \bottomrule
  \end{tabular}
  \label{tab:costs}  
\end{table}

\section{Numerical results}\label{sec:numerical_results}

To assess the proposed stage-aligned RK method, we study a widely used compartmental model. We compare the accuracy of the standard Lagrangian RK method, the auxiliary Euler approach~\cite{kuhn_assessment_2021}, the proposed stage-aligned RK method, and a hybrid adaptation where, independent of the RK scheme for the aggregated system, traveler states are obtained by a simple algebraic scaling based on the initial traveler share. Furthermore, we assess the computation times of these approaches as a function of the number of compartments and patches.

\subsection{Description of the considered model}

We use the different numerical schemes to study an SEIR-type metapopulation model. We consider age-specific dynamics, and use $i \in \{1,\dots,N_{G}\}$ to index the $N_{G}$ distinct age groups. To describe the local transition dynamics, we drop the patch index for clarity. The SEIR equations for age group $i$ are given by
\begin{align}
\begin{aligned}
\label{eq:basicSEIR}
  \frac{\d S_i(t)}{\d t} &= -\,\lambda_i(t)\,S_i(t), \\
  \frac{\d E_i(t)}{\d t} &= \lambda_i(t)\,S_i(t)\;-\;\frac{1}{T_{E_i}}\,E_i(t), \\
  \frac{\d I_i(t)}{\d t} &= \frac{1}{T_{E_i}}\,E_i(t)\;-\;\frac{1}{T_{I_i}}\,I_i(t), \\
  \frac{\d R_i(t)}{\d t} &= \frac{1}{T_{I_i}}\,I_i(t).
\end{aligned}
\end{align}
Here, $S_i(t)$, $E_i(t)$, $I_i(t)$, and $R_i(t)$ denote the corresponding compartment sizes of age group $i$ at time $t$ and $N_i(t) = S_i(t) + E_i(t) + I_i(t) + R_i(t)$ is the total population of age group $i$. The parameters $T_{E_i}$ and $T_{I_i}$ denote the average latency and infectious periods, respectively. The force of infection $\lambda_i(t)$ for age group $i$ is defined as
\begin{align*}
  \lambda_i(t) =\rho_i\, \sum_{j=1}^{N_{G}} \phi_{i,j}(t)\,\frac{I_j(t)}{N_j(t)},
\end{align*}
where $\rho_i \in [0,1]$ for $i=1,\dots,N_G$ denotes the transmission probability per contact for age group $i$, and $\phi(t)\in\mathbb{R}^{N_G \times N_G}_{\geq 0}$ is the contact matrix with entries $\phi_{ij}(t)$ for $i,j=1,\dots,N_G$ representing the average number of daily contacts that an individual of age group $i$ has with individuals of age group $j$. The number of compartments in the resulting model is $N_C=4N_G$. 

\subsection{Limitations of traveler computation with auxiliary Euler heuristic}
\label{sec:euler_approximation_traveler}

The previously suggested heuristic~\cite{kuhn_assessment_2021} updates traveler states through an auxiliary explicit Euler step that is decoupled from the patch ODE solver. This method updates each subpopulation currently present in a patch by a single step using patch-level quantities (e.g., force of infection) computed from the aggregated compartments at the beginning of the step (cf.~\cref{eq:auxeuler}) introducing specific numerical limitations.

In addition to the limited convergence order of this auxiliary Euler method, the auxiliary Euler step approach can lead to overshooting in traveler updates. Overshooting approximations result in negative values for the remaining local population in the patch, meaning that the updated subpopulation state exceeds the simulated aggregated total, i.e., $\widetilde{x}^{(q)}_j(t_{a+1})-x^{(p;q)}_j(t_{a+1})<0$ for at least one compartment $j=1,\ldots,N_C$. This issue requires heuristic corrections, such as subtracting from the largest compartment instead, and it becomes particularly likely when traveler shares in a specific compartment are high in relation to the local population.

\begin{figure}[H]
  \centering
  \includegraphics[width=\textwidth]{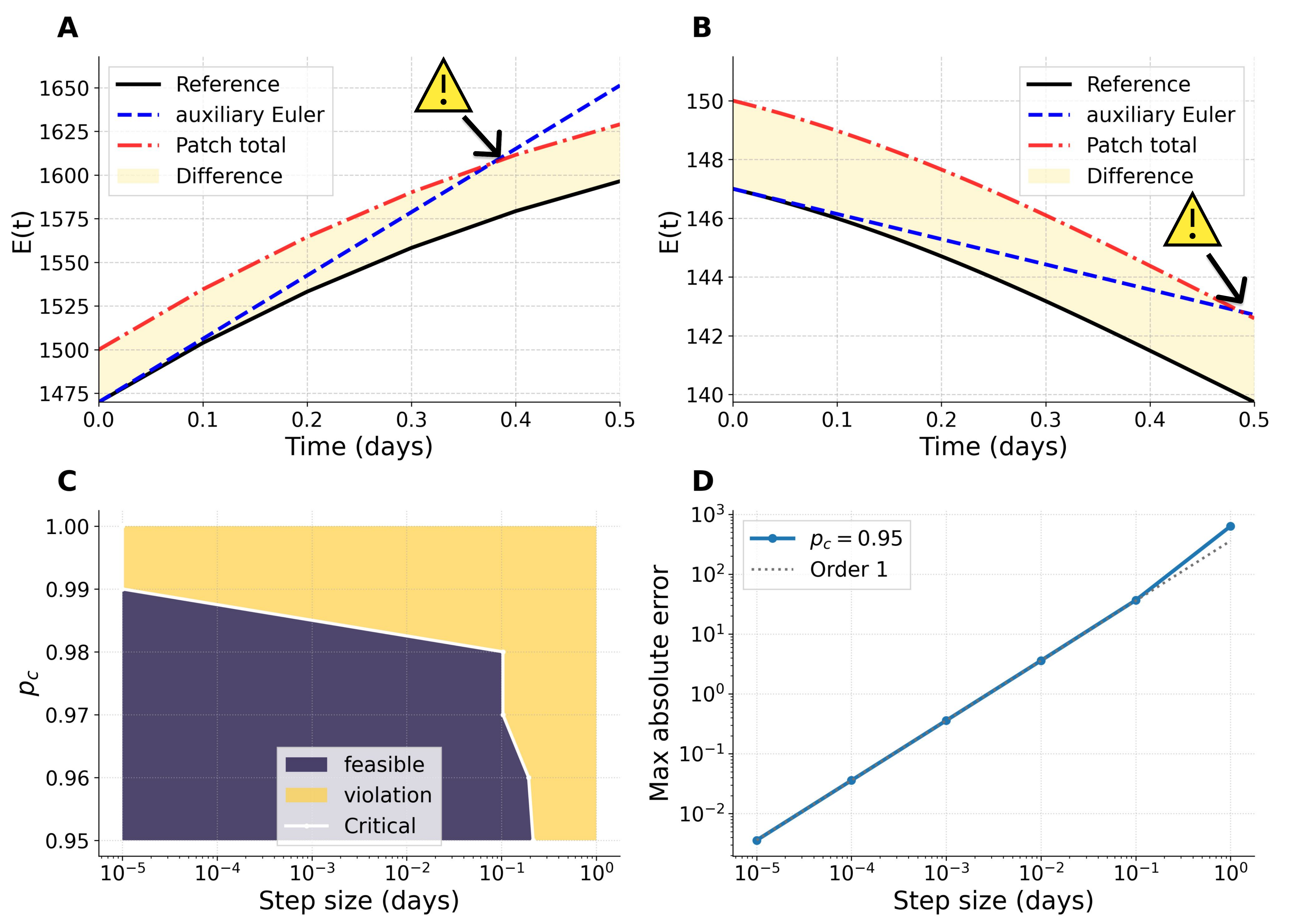}
  \caption{\textbf{Implications for subpopulation updates through the auxiliary Euler step as in~\cite{kuhn_assessment_2021}.}
    \textbf{(A–B)} One-step tests demonstrating the evolution of the exposed subpopulation $E(t)$ under two different initial conditions ($t_{\max}=h = 0.5$\,days, $p_{c,S}=p_{c,E}=p_{c,I}=p_{c,R}=0.98$): the solid black curve is the subpopulation reference obtained from the standard Lagrangian model with a fine step size; the blue dashed curve is the single-step auxiliary Euler update; the red dash-dotted curve is the simulated patch total. Yellow shading highlights the gap between the subpopulation and the total, and warning icons indicate Euler overshoot of the patch total.
    \textbf{(C)} Feasibility map across a uniform traveler fraction $p_c$ (applied equally to S,E,I,R) and step size $h$. Dark cells indicate fewer than $2\%$ violating time points (negativity or overshoot). Yellow cells indicate at least $2\%$ violating time points.
    \textbf{(D)} Convergence at $p_c=0.95$: The maximum error over time and all compartments decreases approximately linearly with~$h$, consistent with first-order behavior (grey dotted slope-1 guideline). Setups are listed in~\cref{tab:setups_params}.}
  \label{fig:euler_compare}
\end{figure}

For the parameter and population values listed in~\cref{tab:setups_params}, \cref{fig:euler_compare} demonstrates this overshooting problem. Short-horizon tests with $t_{\max}=h=0.5$ and large traveling fractions ($p_{c,S}=p_{c,E}=p_{c,I}=p_{c,R}=0.98$) directly lead to overshooting approximations (\cref{fig:euler_compare}A,B). Note that while traveler shares of more than 90\,\% may seem exceptionally large, they can easily occur if a specific compartment in a destination patch is small compared to the incoming traveler volume.
The feasibility map across a grid of initial fractions and step sizes shows that admissible step sizes decrease rapidly as $p_c\to 1$, whereas moderate traveler fractions allow substantially larger~$h$ without requiring heuristic correction (\cref{fig:euler_compare}C). Finally, compared to a high-accuracy reference solution ($h=10^{-6}$) based on the standard Lagrangian approach, the maximum error of the heuristic approach decreases approximately linearly in~$h$, indicating strict first-order convergence (\cref{fig:euler_compare}D).

\subsection{Numerical validation}
\label{sec:numerical_val}

To validate our findings numerically, we use the standard Lagrangian model from~\cref{eq:MasterEquation_wo_mob} as the ground truth and compare it against the proposed stage-aligned approach. We consider a scenario with a resident population that always stays in Patch~$p$ and one additional group from Patch~$q$, $q\neq p$, that arrives in Patch $p$ at $t=0$ and stays for the entire simulation horizon of $t_{\max}=100$ days. The local dynamics follow the SEIR model (\cref{eq:basicSEIR}) with a single age group ($N_{G}=1$). The initial populations and model parameters are listed in~\cref{tab:setups_params}.

We validate the theory numerically with four explicit Runge-Kutta methods of orders 1, 2, 3, and 4. The specific Butcher tableaus for these methods are given in~\ref{sec:appendix_tableau}. To avoid naming confusion with the auxiliary Euler heuristic, we refer to the explicit Euler integrator as RK-1. As an additional comparison, we consider a hybrid scheme that solves the patch-aggregated system using the high-order RK-4 method but updates traveler states using a simple RK-1 step. This approach is a simple improvement over the heuristic approach of~\cite{kuhn_assessment_2021}; the hybrid approach prevents overshooting by design but does not guarantee traveler state computations as in the standard Lagrangian formulation. \cref{fig:verification} summarizes the findings.

\begin{figure}[!t]
  \centering
  \includegraphics[width=\textwidth]{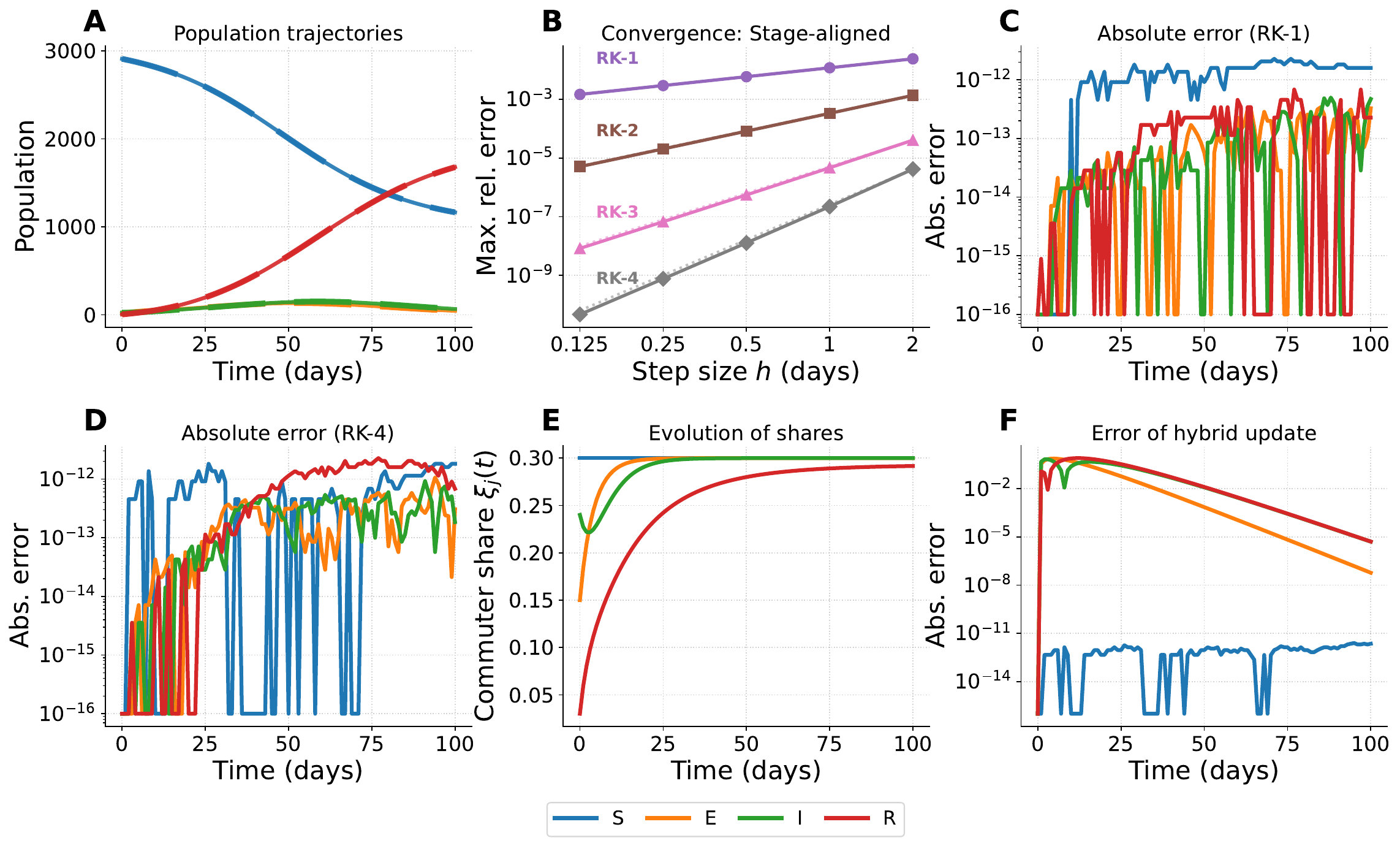}
  \caption{\textbf{Validation of the traveler state computation scheme with convergence plots for different Runge-Kutta methods.}
  \textbf{(A)} Time series of total population dynamics for compartments \(S,E,I,R\) over 100 days (stage-aligned with RK-1 ($h=1$) in solid lines; standard Lagrangian with RK-1 ($h=1$) in dashed lines). 
  \textbf{(B)} Convergence of the maximum relative error for the stage-aligned scheme (RK-1, RK-2, RK-3, and RK-4) using different step sizes $h$. The reference solution is computed with the standard Lagrangian approach using the RK-4 integrator and a step size of $h = 10^{-6}$. Each dotted line shows the theoretical convergence order.
  \textbf{(C)} Absolute error over time for each compartment \(S,E,I,R\) comparing the solution from the stage-aligned approach to the standard Lagrangian solution (both using RK-1 with $h=1$).
  \textbf{(D)} Absolute error over time for each compartment \(S,E,I,R\) comparing the solution from the stage-aligned approach to the standard Lagrangian solution (both using RK-4 with $h=1$).
  \textbf{(E)} Evolution of the traveler shares $\xi^{(p;q)}_j(t)$ for each compartment over time computed using the stage-aligned solution of order 1 as in (A).
  \textbf{(F)} Absolute error over time for each compartment \(S,E,I,R\) when comparing the hybrid approach (RK-4 for the aggregated system and RK-1 for the computation of traveler states) to the standard Lagrangian integrator (RK-4). Here, the step size is fixed at $h=1.0$.}
  \label{fig:verification}
\end{figure}

The stage-aligned approach with a RK-1 scheme yields trajectories indistinguishable from the standard Lagrangian reference model solved with the same method~(\cref{fig:verification}A).

To evaluate the error convergence for the four explicit Runge-Kutta methods, we computed a high-precision reference solution using the standard Lagrangian model with a small step size of $h=10^{-6}$. With decreasing step sizes, the proposed approach converges with its corresponding theoretical order $q \in \{1, 2, 3, 4\}$~(\cref{fig:verification}B). The dotted lines indicate the optimal convergence order, which are even perfectly covered for methods of first and second order. 

Beyond consistency in convergence order, the absolute errors show that the numerical solutions of the traveler groups in the standard Lagrangian approach and the stage-aligned computation are identical up to rounding errors. For both the RK-1 and the RK-4 scheme, the absolute errors between the numerical solutions are close to the rounding error in the beginning and, while accumulating slightly over time, always stay within a range of $10^{-16}$ to $10^{-12}$, thus confirming our findings~(\cref{fig:verification}C,D).

We furthermore observe that the share $\xi^{(p;q)}_j(t)$ (see~\cref{eq:shares}) for the susceptible compartment stays constant~(\cref{fig:verification}E). As no further traveling events occur during the simulation horizon, this compartment without inflows behaves exactly in line with our findings from~\cref{thm:pure_outflow}. Driven by the application, all shares stabilize as the epidemic dynamics approach herd immunity. In contrast, the absolute error of the hybrid strategy (RK-4 for the solution of the aggregated system and an RK-1 step for traveler state computation) is not in the range of the rounding error for compartments with both in- and outflows~(\cref{fig:verification}F).

Summarized, our experiments confirm the theoretical findings. The stage-aligned approach reliably computes traveler state evolutions with the convergence order of the Runge-Kutta scheme used to solve the patch-aggregated formulations. For compartments without inflows (like $S$), the stage-aligned update can be simplified to a direct update with the initial share value without losing accuracy.

\subsection{Assessment of computational efficiency}
\label{sec:numerical_stagealigned}

\begin{figure}[!t]
\centering
\includegraphics[width=1.0\textwidth]{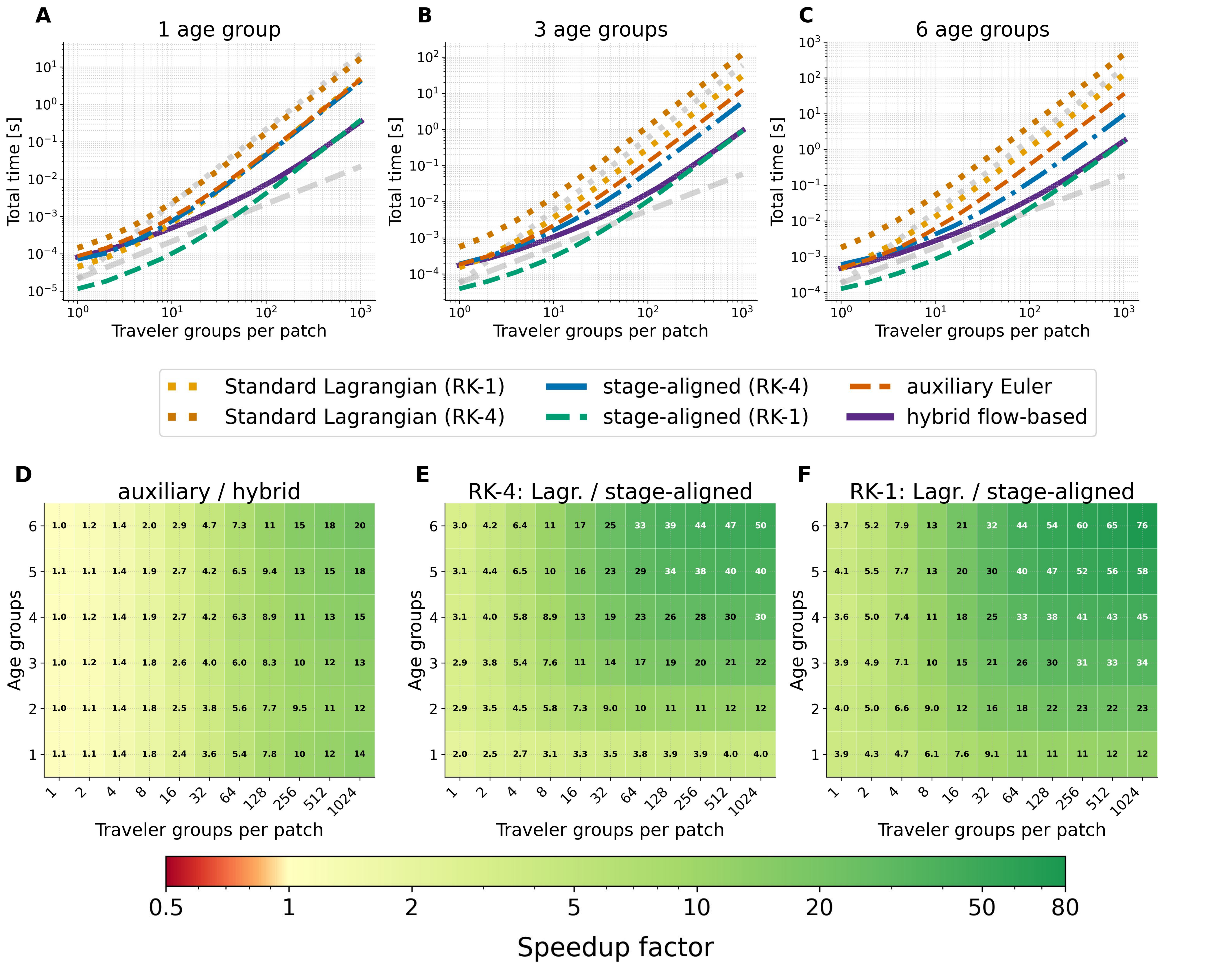}
\caption{\textbf{Computational scaling and speedups for traveler updates using a SEIR model (see~\cref{eq:basicSEIR}) with a step size of $h=0.5$\,days for $50$\,days.}
\textbf{(A–C)} Median runtime per simulation (in seconds) against the number of patches for $1$, $3$, and $6$ age groups.
Lines show the different strategies (standard Lagrangian with RK-1 and RK-4, stage-aligned with RK-1 and RK-4, auxiliary Euler heuristic, hybrid approach) as explained in the legend. Grey reference lines indicate linear $\mathcal{O}(N_P)$ (dashed) and quadratic $\mathcal{O}(N_P^2)$ (dotted) trends.
\textbf{(D–F)} Speedup heatmaps across patches (rows) and age groups (columns).
Each cell displays the ratio of median runtimes for the indicated pair: \textbf{(D)} auxiliary Euler heuristic / hybrid approach, \textbf{(E)} standard Lagrangian (RK-4) / stage-aligned (RK-4), \textbf{(F)} Standard Lagrangian (RK-1) / stage-aligned (RK-1).
The color scale is logarithmic with a break at 1 (yellow to green~$>$~1: suggested approach faster; red to orange~$<$~1: existing approach faster).
Timings are Google Benchmark medians reported as seconds per simulated day. For the parameters and populations used in these simulations see~\cref{tab:setups_params}.}
\label{fig:benchmark_ags}
\end{figure}

As a final assessment, we evaluate the computational efficiency of the proposed stage-aligned approach. All numerical schemes were implemented in C++ as part of the MEmilio framework~\cite{bicker2026memilio,Bicker_MEmilio_v2_1_0}. Benchmarks were executed on an Intel Xeon ``Skylake'' Gold 6132 (2.60~GHz) using four nodes with 14 CPU cores each and 384~GB DDR4 memory. We compare four distinct strategies: i) the standard Lagrangian formulation (solved with RK-1 and RK-4 methods), ii) the novel stage-aligned approach (with RK-1 and RK-4 methods), iii) the auxiliary Euler heuristic introduced in~\cite{kuhn_assessment_2021} (with RK-4 for the aggregated system and an external auxiliary Euler step for travelers), and iv) the hybrid approach (with RK-4 for the aggregated system and an RK-1 step for the traveler state computation). 

All simulations use the SEIR model defined in~\cref{eq:basicSEIR}, solved numerically with a fixed step size of $h=0.5$ days on a simulation horizon of $50$ days. As in the validation scenario, traveling subpopulations arrive at their destination patches at $t=0$ and stay there for the entire simulation horizon. The parameters and initial conditions are presented in~\cref{tab:setups_params}, with the total population distributed uniformly across the age groups. We vary the problem size by increasing the number of traveler groups per patch, which corresponds to $N_P-1$, and the number of demographic stratifications ($N_G$) in a fully connected setting and report Google Benchmark medians as seconds per simulated day.

The assessment of the computation times (\cref{fig:benchmark_ags}A–C) shows that RK-1 and RK-4 for standard Lagrangian models scale quadratically with the number of patches and linearly with the number of demographic stratifications ($N_G$ with $N_C=4N_G$). This scaling behavior is driven by the number of state variables and the complexity of the numerical scheme. Since the computational cost of explicit Runge-Kutta methods is linear with respect to the number of ODEs, the solver runtime scales directly with the $\mathcal{O}(N_P^2 N_C)$ system dimension, which is confirmed by our benchmark results.

The stage-aligned realizations partially decouple the ODE integration dimension from the network connectivity, solving only for the patch-aggregated compartments in a system size of $\mathcal{O}(N_P N_C)$. Our results demonstrate that the remaining algebraic updates for the $\mathcal{O}(N_P^2 N_C)$ traveler groups are computationally cheap compared to the expensive evaluations of the coupled right-hand sides (e.g., force of infection) required by the ODE solver. The total computational complexity is thus composed of a linear term for ODE integration and a quadratic term for algebraic computation.
In the observed range, the linear cost of integration dominates the runtime for small to medium network sizes, while at larger network sizes (e.g., $N_P>100$), the quadratic cost of the traveler state computation becomes more visible, leading to a transition from linear to quadratic scaling behavior, however, with substantially reduced runtime with respect to the standard Lagrangian approach.
Comparing the efficiency of the stage-aligned and hybrid strategies reveals distinct performance characteristics. For larger network sizes, the hybrid approach and the stage-aligned RK-1 approach are the fastest. As the number of patches increases towards $1025$ (with 1024 local travel connections), the runtimes of the stage-aligned RK-1 and hybrid variants align. This indicates that for very large networks, the computational load shifts towards the algebraic computation steps shared by both methods, thereby decreasing the relative weight of the ODE integrator. The stage-aligned RK-4 method incurs a moderate overhead due to the multiple internal stages required for higher-order accuracy, while the stage-aligned Euler variant proves to be the fastest strategy overall. Nevertheless, the stage-aligned RK-4 method matches or even exceeds the performance of the auxiliary Euler heuristic -- particularly for higher stratifications ($N_G=6$) -- while inherently preventing overshooting and ensuring higher order numerical approximations. 

Quantifying the relative speedups reveals that the novel methods outperform the previously available approaches in 194 out of 198 configurations (with parity only in two-patch setups with a single traveling subpopulation per patch). For larger networks, the performance improvements of the novel methods becomes substantial.
Comparing the auxiliary Euler heuristic to the hybrid approach shows that the hybrid approach is up to $20$ times faster~(\cref{fig:benchmark_ags}D), while additionally preventing overshooting and being exact for compartments without inflows.
Furthermore, the stage-aligned RK-4 variant yields massive performance gains over the standard Lagrangian RK-4 model, being up to $50$ times faster at the largest scale, while obtaining the same numerical result~(\cref{fig:benchmark_ags}E).
Similarly, the stage-aligned RK-1 variant is up to $76$ times faster than the standard Lagrangian model solved with a RK-1 method~(\cref{fig:benchmark_ags}F). These advantages are driven by the piecewise-continuous formulation limiting the ODE system size to $\mathcal{O}(N_P\,N_C)$, while the traveler states are updated via a series of efficient algebraic operations.

While the hybrid approach outperforms the prior heuristic method and prevents overshooting by design, it trades performance against accuracy by only using a first-order scheme for the travelers and, thus, potentially also affecting non-traveler accuracy. When first-order accuracy is sufficient, the stage-aligned RK-1 method should be preferred. If accuracy beyond first order is required, the stage-aligned variant with Runge-Kutta schemes of order 2, 3, or more can be chosen easily.

The absolute runtime for the standard Lagrangian model using the RK-4 integrator for the considered simulation scenario takes approximately $472$ seconds for the largest network (1025 patches and six age groups). In direct comparison, the stage-aligned RK-4 variant requires only $9.5$ seconds for the same simulation. While the single-run time of the standard model might appear manageable, the practical implications become evident in validation and calibration tasks. Modeling often requires Bayesian parameter inference, sensitivity analysis, or ensemble forecasting involving tens of thousands of model evaluations.

\section{Conclusion}
\label{sec:conclusion}

In this work, we presented a numerically efficient method for computing traveler states in a metapopulation setting. We considered the limit case of a Lagrangian formulation with instantaneous mobility exchanges at arbitrary, user-defined mobility event timings and a recently suggest graph-based approach. By solving the dynamics of the patch-aggregated systems, the dimension of the global ODE system is reduced substantially when compared of the standard Lagrangian setting, i.e., it only scales linearly with the number of patches, while, in a densely or fully connected network, the standard Lagrangian formulation has quadratic scaling. The computation of traveler states still scales quadratically with the number of patches, but the associated algebraic updates are much less costly. Reusing relevant function evaluations from the Runge-Kutta stages of the aggregated solution and aligning the traveler state computation with these stages avoids a large number of numerical operations. With these updates, our approach resolves issues of the previously suggested heuristic for traveler state estimation. We proved that the numerical solution obtained from the novel stage-aligned approach is identical to the solution of the standard Lagrangian formulation if the same Runge-Kutta methods are used. In particular, our novel method inherits the convergence order of the used Runge-Kutta scheme. For compartments without inflows, we showed that we can furthermore simplify the computation scheme and obtain the identical solution by adequate rescaling of the aggregated solution. The proposed method is well suited for a broad class of metapopulation models with explicit mobility (i.e., where individuals are physically moved between patches rather than just implicitly coupled via contact networks) that admit a Lagrangian or Lagrangian-type formulation. Our construction is globally Lagrangian, while locally it computes state transitions based on patch-aggregated dynamics.

In our benchmarks, we showed that the novel stage-aligned approach substantially outperforms standard Runge-Kutta methods for the Lagrangian formulation for models with demographic stratification into one, three, and six age groups and increasing numbers of patches between 2 and 1025. Although the quadratic traveler state computation eventually dominates the runtime for very large networks in the novel formulation, it remains computationally much lighter than the standard Lagrangian approach. For a network with $1025$ patches, i.e., 1024 local traveler connections, and six age groups, we obtain a speedup of up to $50$-fold for the stage-aligned approach using a RK-4 scheme and up to $76$-fold using a RK-1 scheme relative to the standard Lagrangian formulation. Even for networks with 65 and 257 patches, the performance gain is substantial, reaching factors of approximately $33$ to $44$ for RK-4 and $44$ to $60$ for RK-1, respectively.

Overall, the proposed stage-aligned formulation enables accurate and scalable simulation of metapopulation dynamics with explicit mobility. It preserves the numerical properties of the underlying Runge–Kutta discretization while substantially reducing the computational burden relative to the standard Lagrangian formulation. This considerably broadens the range of large-scale applications for which Lagrangian-type models become computationally feasible.

\appendix

\section{Initialization SEIR model}
\label{sec:appendix_init_seir}

\FloatBarrier
\begin{table*}[!h]
\small
\centering
\caption{\textbf{Overview of model and setup parameters across the numerical subsections.} In case of more than one age group, the total population is distributed uniformly across the groups.}
\label{tab:setups_params}
\setlength{\tabcolsep}{3pt}
\renewcommand{\arraystretch}{1.08}
\begin{tabularx}{\textwidth}{L{1.6cm} D V V V V}
\toprule
\multirow{3}{*}{\textbf{Symbol}} & \multirow{3}{*}{\textbf{Description}} &
\multicolumn{2}{c}{\textbf{Comparison to prior work}} &
\multicolumn{1}{c}{\textbf{Verification}} &
\multicolumn{1}{c}{\textbf{Performance}} \\
\cmidrule(lr){2-3}\cmidrule(lr){4-4}\cmidrule(lr){5-6}
& & \multicolumn{2}{c}{\cref{sec:euler_approximation_traveler}}  & \cref{sec:numerical_val} & \cref{sec:numerical_stagealigned} \\
& & \cref{fig:euler_compare}A & \cref{fig:euler_compare}B &  \cref{fig:verification} &\cref{fig:benchmark_ags} \\
\midrule
$T_E$  & Latency period (days) &  1.0 & 1.0 & 5.2 & 5.2 \\
\rowrule
$T_I$  & Infectious period (days) &  1.0 & 1.0 & 6.0 & 6.0 \\
\rowrule
$\rho$ & Transmission probability per contact &  0.1 & 0.1 & 0.1 & 0.1 \\
\rowrule
$\phi$ & Mean daily contacts (contact matrix) & 2.7 & 2.7 & 2.7  & $\mathbf{1}_{N_G \times N_G}$ \\
\rowrule
$t_{\max}$  & Simulation end time (days) & 0.5 & 0.5 &  100 &  50 \\
\rowrule
$h$  & Step size & 0.5 & 0.5 & \{2.0, 1.0, 0.5, 0.25, 0.125\} & 0.5 \\
\rowrule
$N_{G}$ & Number of age groups & 1 & 1 & 1  & \(\{1,2,3,4,5,6\}\)\\
\rowrule
$N_P$ & Number of patches  & 1 & 1 & 1  &\(\{2,3,5,\dots,1025\}\) \\
\rowrule
$N_c$ & Number of local traveler groups& 1 & 1 & 1  & \(\{1,2,4,\dots,1024\}\) \\
\rowrule
$S_0$ & Initial susceptible (total) & 5000 & 9350 & 9700  & 9700 \\
\rowrule
$E_0$ & Initial exposed (total) & 1500 & 150 & 100  &100 \\
\rowrule
$I_0$ & Initial infected (total) & 1500 & 120 & 100  & 100 \\
\rowrule
$R_0$ & Initial recovered (total) & 2000 & 180 & 100  & 100 \\
\rowrule
$p_{c,S}$ & Fraction of initial susceptible travelers &  0.98 & 0.98 & 0.3 &  $0.1 \times N_C$ \\
\rowrule
$p_{c,E}$ & Fraction of initial exposed travelers &  0.98 & 0.98 & 0.15 & $0.1 \times N_C$ \\
\rowrule
$p_{c,I}$ & Fraction of initial infected travelers &  0.98 & 0.98 & 0.24 &  $0.1 \times N_C$ \\
\rowrule
$p_{c,R}$ & Fraction of initial recovered travelers & 0.98 & 0.98 &  0.03 &   $0.1 \times N_C$ \\
\bottomrule
\end{tabularx}
\end{table*}

 \FloatBarrier
\section{Butcher Tableaus of the Numerical Solvers}
\label{sec:appendix_tableau}
 \FloatBarrier
For the numerical validation in~\cref{sec:numerical_val}, we utilized the following explicit Runge-Kutta methods defined by their Butcher tableaus $(c, A, b^T)$:

\begin{table}[!h]
    \centering
    \begin{minipage}{0.45\textwidth}
        \centering
        \textbf{RK-1 (Explicit Euler)} \\[0.2cm]
        \begin{tabular}{c|c}
            0 & 0 \\
            \hline
              & 1
        \end{tabular}
    \end{minipage}
    \hfill
    \begin{minipage}{0.45\textwidth}
        \centering
        \textbf{RK-2 (Midpoint method)} \\[0.2cm]
        \begin{tabular}{c|cc}
            0   & 0   & 0 \\
            1/2 & 1/2 & 0 \\
            \hline
                & 0   & 1
        \end{tabular}
    \end{minipage}
    
    \vspace{1cm}
    
    \begin{minipage}{0.45\textwidth}
        \centering
        \textbf{RK-3} \\[0.2cm]
        \begin{tabular}{c|ccc}
            0   & 0   & 0   & 0 \\
            1/2 & 1/2 & 0   & 0 \\
            1   & -1  & 2   & 0 \\
            \hline
                & 1/6 & 4/6 & 1/6
        \end{tabular}
    \end{minipage}
    \hfill
    \begin{minipage}{0.45\textwidth}
        \centering
        \textbf{RK-4} \\[0.2cm]
        \begin{tabular}{c|cccc}
            0   & 0   & 0   & 0 & 0 \\
            1/2 & 1/2 & 0   & 0 & 0 \\
            1/2 & 0   & 1/2 & 0 & 0 \\
            1   & 0   & 0   & 1 & 0 \\
            \hline
                & 1/6 & 1/3 & 1/3 & 1/6
        \end{tabular}
    \end{minipage}
\end{table}
 \FloatBarrier

\section*{Acknowledgements}
This work was supported by the Initiative and Networking Fund of the Helmholtz Association (grant agreement number KA1-Co-08, Project LOKI-Pandemics). It was furthermore supported by the German Federal Ministry of Education and Research and the German Federal Ministry of Research, Technology and Space under grant agreement 031L0325A (Project TwinChain) and the Deutsche Forschungsgemeinschaft (DFG, German Research Foundation) (grant agreement 528702961). TwinChain is part of the Modeling Network for Severe Infectious Diseases (MONID). Additionally, this work was supported by the European Union via the ERC grant INTEGRATE, grant agreement number 101126146, and under Germany’s Excellence Strategy by the Deutsche Forschungsgemeinschaft (DFG, German Research Foundation) (EXC 2047—390685813, EXC 2151—390873048, and 524747443), the University of Bonn via the Schlegel Professorship of J.H. 

\section*{Competing interests}
The authors declare that they have no known competing financial interests or personal relationships that could have appeared to influence the work reported in this paper.

\section*{Data availability}
All simulations were implemented and performed using the MEmilio framework \cite{bicker2026memilio,Bicker_MEmilio_v2_1_0}. The specific simulation scenarios and plotting scripts are accessible at \url{https://github.com/SciCompMod/memilio-simulations}.

\section*{CRediT authorship contribution statement}
\begin{itemize}
    \item[] \textbf{Henrik Zunker:} Conceptualization, Data curation, Formal analysis, Investigation, Methodology, Software, Validation, Visualization, Writing – original draft, Writing – review \& editing 
    \item[] \textbf{Ren\'e Schmieding:} Conceptualization, Software, Validation, Writing – review \& editing 
    \item[] \textbf{Jan Hasenauer:} Conceptualization, Formal analysis, Funding acquisition, Methodology, Project administration, Resources, Supervision, Validation, Writing – review \& editing
    \item[] \textbf{Martin J. Kühn:} Conceptualization, Formal analysis, Funding acquisition, Investigation, Methodology, Project administration, Resources, Supervision, Validation, Writing – original draft, Writing – review \& editing 
\end{itemize}

\bibliographystyle{elsarticle-num}
\bibliography{literature.bib}

\end{document}